\documentclass[a4paper,12pt]{article}
\usepackage{amsthm}
\usepackage{amsfonts}
\usepackage{amsmath}
\usepackage{amssymb}
\usepackage[pdftex]{graphicx}
\usepackage{hyperref}
\usepackage{datetime}
\usepackage{diagrams}
\usepackage{commath}
\usepackage{float}
\restylefloat{table}

\newtheorem{theorem}{Theorem}[section]
\theoremstyle{definition}

\theoremstyle{remark}

\theoremstyle{plain}
\newtheorem{lemma}[theorem]{Lemma}
\newtheorem{corollary}{Corollary}[section]

\newtheorem{proposition}{Proposition}[section]

\newcommand{\C}{\mathbb{C}}
\newcommand{\Z}{\mathbb{Z}}

\newcommand{\diag}{\text{diag}}
\newcommand{\bq}{/\!\!/}

\begin{document}

\begin{center} \textbf{The classification and curvature of biquotients of the form $Sp(3)\bq Sp(1)^2$}

\bigskip

Jason DeVito, Robert DeYeso III, Michael Ruddy, Philip Wesner \end{center}

\begin{abstract}
We show there are precisely $15$ inhomogeneous biquotients of the form $Sp(3)\bq Sp(1)^2$ and show that at least $8$ of them admit metrics of quasi-positive curvature.

\end{abstract}

\section{Introduction}
\label{intro}

Manifolds which admit metrics of positive sectional curvature have been studied since the inception of Riemannian geometry, but despite this, very few examples are known.  There are many examples of non-negatively curved manifolds, however, leading one to expect there to be obstructions to pass from non-negative curvature to positive curvature.  Unfortunately, the only known obstructions depend on the fundamental group and, in particular, vanish for simply connected manifolds.

As a means to better understand the difference between these two classes of manifolds, attention has turned towards quasi-positively and almost positively curved Riemannian manifolds. A non-negatively curved Riemannian manifold is said to be quasi-positively curved if it has a point for which all $2$-planes have positive sectional curvature.  A Riemannian manifold is called almost positively curved if the set of points for which all planes are positively curved is dense.

The first known example of a quasi-positively curved manifold was an exotic sphere found by Gromoll and Meyer \cite{GrMe1}.  Since then, many new examples of both quasi-positively curved and almost positively curved manifolds have been found by Wilking \cite{Wi}, Petersen and Wilhelm\cite{PW2,W}, Eschenburg and Kerin \cite{EK,Ke1,Ke2}, Kerr and Tapp \cite{KT,Ta1}, and the first author\cite{D1}.

The majority of these previous results are constructed via Riemannian submersions onto biquotients.  A biquotient is any manifold which is diffeomorphic to the quotient of a homogeneous space $G/H$ by an effectively free isometric action of a subgroup $K\subseteq G$.  Alternatively, if $f:U\rightarrow G\times G$ is a homomorphism, then this defines an action of $U$ on $G$ by $f(h)\ast g = (h_1,h_2)\ast g = h_1 gh_2^{-1}$.  When the action is effectively free, the orbit space naturally has the structure of a manifold, denoted $G\bq U$, and is called a biquotient.  When the image of $U$ is a subgroup of $\{e\}\times G$, the action is automatically effectively free and the quotient is the homogeneous space $G/(U/\ker(f))$.

Much is already known about curvature on biquotients of the form $G\bq U$ with $G$ of rank less than $3$, so it is natural to search for examples among rank $3$ groups.  We prove the following:

\begin{theorem}\label{main} There are precisely $15$ inhomogeneous effectively free biquotient actions of $Sp(1)^2$ on $Sp(3)$, giving rise to manifolds distinct up to diffeomorphism.  At least $8$ of these admit metrics of quasi-positive curvature.

\end{theorem}

More precisely, we find a non-negatively curved metric on $Sp(3)$ which simultaneously induces quasi-positive curvature on the biquotients $N_1$ through $N_8$ in Table \ref{table:emb} on page \pageref{table:emb}.

By way of comparison, there are precisely $4$ homogeneous actions of $Sp(1)^2$ on $Sp(3)$.  Kerr and Tapp \cite{KT,Ta1} have shown that three of the four resulting quotients admit metrics of quasi-positive curvature, and the first author \cite{D1} has shown that one of these three admits a metric of almost positive curvature.  We do not know if our examples are almost positively curved, nor do we know if the remaining $7$ biquotients and $1$ homogeneous space admit metrics of quasi-positive curvature.

The paper is organized as follows:  Section 2 reviews preliminary information about representation theory along with the geometry and topology of biquotients.  Section 3 classifies all effectively free biquotients of $Sp(1)^2$ on $Sp(3)$, see Theorem \ref{biqclass} and Table \ref{table:emb}.  In Section 4 we construct metrics of quasi-positive curvature on the $8$ new examples, see Theorems \ref{n4n5}, \ref{n1n2n3n6}, and \ref{n7n8}.  Finally, in Section 5 we compute the cohomology groups of all the biquotients in Table \ref{table:emb}.  It turns out, the order of the $8$th cohomology group (Table \ref{table:det}) distinguishes most of them, while the remaining biquotients have distinct first Pontryagin classes, see Table \ref{table:Pontryagin}.

We are grateful to the anonymous referee for suggesting several improvements.

\section{Preliminaries}\label{Prelim}

As mentioned in the introduction, given compact Lie groups $U$ and $G$, any homomorphism $f = (f_1,f_2):U\rightarrow G\times G$ gives rise to a natural $U$ action on $G$ given by $ u\ast g = f_1(u)\, g \, f_2(u)^{-1}$.  When the action is effectively free, the orbit space, called a biquotient and denoted $G\bq U$, naturally has the structure of a smooth manifold for which the canonical projection $\pi:G\rightarrow G\bq U$ is a smooth submersion.  Biquotients were systematically studied in Eschenburg's Habilitation \cite{Es2}.

A simple criterion to determine when an action is effectively free is given by the following proposition.

\begin{proposition}\label{free}  Writing $f_i(u) = u_i$, a biquotient action of $U$ on $G$ is effectively free if and only if for any $(u_1,u_2)\in f(U)$, if $u_1$ is conjugate to $u_2$ in $G$, then $u_1 = u_2 \in Z(G)$.  Such an action is free if and only if for any $(u_1,u_2)\in f(U)$, if $u_1$ is conjugate to $u_2$ in $G$, then $u_1 = u_2 = e$.

\end{proposition}

Since every element of a Lie group $U$ is conjugate to an element in its maximal torus $T_U$, this proposition implies that the $U$ action on $G$ is (effectively) free if and only if the induced action of $T_U$ on $G$ is (effectively) free.

To begin classifying biquotient actions, we note that only the image of $f$ matters when determining the biquotient action.  In fact, the action is determined, up to equivalence, by the conjugacy class of the image of $f$.

\begin{proposition}\label{equiv}

Suppose $f, f':U\rightarrow G\times  G$ are homomorphisms and that the images are conjugate: $f'(u) = (g_1,g_2)f(u) (g_1,g_2)^{-1}$ for all $u\in U$. Then the map $\phi:G\rightarrow G$, defined by $\phi(g) = g_1\, g \, g_2^{-1}$, is an equivariant diffeomorphism between the two actions on $G$.  

\end{proposition}

\begin{proof}  Let $u = (u_1, u_2)\in f(U)$.  Then \begin{align*}\phi(u\ast g) &= \phi(u_1\, g\, u_2^{-1})\\ &= g_1\, u_1\, g\, u_2^{-1} \, g_2^{-1}\\ &=  (g_1\, u_1\, g_1^{-1})\, g_1 \, g\, g_2^{-1}\, (g_2 \, u_2^{-1}\, g_2^{-1})\\ &= f'(u)\ast (g_1 \, g\, g_2^{-1})\\ &= f'(u)\ast \phi(g),\end{align*} so $\phi$ is an intertwining map.  Furthermore, the inverse is clearly given by $\phi^{-1}(g) = g_1^{-1} \, g\, g_2^{-1}$, so $\phi$ is an equivariant diffeomorphism.

\end{proof}

In particular, the $U$ action on $G$ via $f$ is (effectively) free if and only if the action via $f'$ is (effectively) free and the two quotients are canonically diffeomorphic.  It follows that we may classify all biquotient actions of $U$ on $G$ by classifying the conjugacy classes of images of homomorphisms from $U$ into $G\times G$ and then checking each of these to see if the induced action is effectively free.  Combining this with Proposition \ref{free}, it follows that if $f(T_U)\subseteq T_{G\times G}$, then the action of $U$ on $G$ is (effectively) free if and only if the induced action of $T_U$ on $T_{G\times G}$ is (effectively) free.

We further point out that in the case of $G = Sp(3)$ we can take $T_G$ as the set of diagonal matrices with complex entries of length $1$.  Two such matrices are conjugate in $Sp(3)$ if and only if the entries are the same up to reordering and complex conjugation.

\subsection{Representation theory}
\label{Reptheory}

The primary tool we have for constructing homomorphisms $f:U\rightarrow G^2$ is representation theory.  All of the following information can be found in \cite{FH}.  Recall that a representation of $U$ is a homomorphism $\rho:U\rightarrow Gl(V)$ for some complex vector space $V$.  It is well known that if $U$ is a compact semi-simple Lie group, then $\rho(U)$ is conjugate to a subgroup of $SU(V)$ and that $\rho$ is completely reducible -- every such $\rho$ is a direct sum of irreducible representations.  Furthermore, when $V = \C^{2n}$ we say that the representation $\rho$ is symplectic if the image of $\rho$ is conjugate to the natural subgroup $Sp(n)\subseteq SU(2n)$.

Recall the following well known proposition.

\begin{proposition}\label{symp} A representation $\rho$ of $U$ is symplectic if and only if $\rho \cong \bigoplus_i (\psi_i\oplus \overline{\psi}_i)\oplus\ \bigoplus_j \phi_j,$ where each $\phi_j$ is symplectic and $\overline{\psi}_i$ denotes the conjugate representation of $\psi_i$.

\end{proposition}

A similar proposition is true for orthogonal representations -- those whose image is conjugate to the natural subgroup $SO(n)$ in $SU(n)$.

Since we are interested in the case $U = Sp(1)^2$, we note that the irreducible representations of a product of compact Lie groups are always given as outer tensor products of irreducible representations of the factors.  We also recall that an outer tensor product of two irreducible representations is symplectic if and only if one of the representations is symplectic and the other is orthogonal, and is orthogonal if and only if they are both symplectic or both orthogonal.

The irreducible representations for every compact simple Lie group have been completely classified.  For $Sp(1)$, we have the following proposition.

\begin{proposition}\label{sp1irrep} For each $n \geq 1$, $Sp(1)$ has a unique irreducible representation of dimension $n$.  When $n$ is even this representation is symplectic, and when $n$ is odd this representation is orthogonal.

\end{proposition}

In particular, every representation of $Sp(1)$ is either symplectic or orthogonal.  This, in turn, implies that in the case of $U = Sp(1)^2$, every irreducible representation is either symplectic or orthogonal.  Since orthogonal representations are equivalent to their conjugates, Proposition \ref{symp} implies that a representation of either $Sp(1)$ or $U$ is symplectic iff every irreducible orthogonal subrepresentation appears with even multiplicity.

Finally, we have a theorem due to Mal'cev \cite{Ma} which connects representation theory with Proposition \ref{equiv}.

\begin{theorem}\label{Mal'cev}Let $f_1, f_2:H\rightarrow G$ with $$G\in\{ SU(n), U(n), Sp(n), SO(n)\}$$ be interpreted as an $n$-dimensional complex, symplectic, or orthogonal representation.  If the representations are equivalent, then the images are conjugate in $G$ except when $G = SO(n)$ and $n$ is even.  In this case, the images are conjugate in $O(n)$ and conjugate in $SO(n)$ if and only if there is at least one irreducible factor of odd dimension.

\end{theorem}

In particular, for $G = Sp(3)$ this theorem tells us that if two subgroups are not conjugate, then the corresponding representations cannot be equivalent.  Hence, to classify the embeddings of $Sp(1)^2$ into $Sp(3)$ up to conjugacy, we begin with the representation theory problem of classifying $3$-dimensional symplectic representations of $Sp(1)^2$.

\subsection{The geometry of biquotients}
\label{geo}

There are two primary methods for constructing metrics on biquotients: Cheeger deformations \cite{Ch1} and Wilking's doubling trick \cite{Wi}.  The starting point for both is the well known observation, due to O'Neill \cite{On1}, that Riemannian submersions are curvature non-decreasing.  Moreover, it is also known that if a compact Lie group $U$ acts isometrically and effectively freely on a manifold $M$, then the orbit space inherits a canonical smooth structure and Riemannian metric for which the natural projection is a Riemannian submersion. For our purposes we will always have that $M$ is a Lie group equipped with a Riemannian metric of non-negative sectional curvature, and hence all of our biquotients will automatically be non-negatively curved.

We now describe Cheeger deformations and Wilking's doubling trick following the exposition in \cite{Ke2}.  For $G$ a compact Lie group and $K\subseteq G$ a closed subgroup, we equip $G$ with a left invariant, right $K$-invariant metric $g$.  For each $t > 0$, $K$ acts on $(G\times K, g+tg|_K)$ isometrically and freely via the action $h\ast(g,k) = (gh^{-1}, hk)$.  It is clear that the map $\psi:G\times K\rightarrow G$ given by $\psi(g,k) = gk$ is $K$-invariant and descends to a diffeomorphism $G\times_K K\rightarrow G$.  Transporting the submersion metric on $G\times_K K$ to $G$ via $\psi$, we obtain a new metric $g_1$ on $G$.

The isometric action of $G$ on $G\times\{e\}\subseteq G\times K$ given by left multiplication commutes with the $K$ action on $G\times K$, and hence descends to a transitive isometric action of $G$ on $(G,g_1)$, so $g_1$ is left invariant.  Likewise, the isometric action of $K$ on $\{e\}\times K\subseteq G\times K$ given by right multiplication descends to an isometric action on $(G,g_1)$, so $g_1$ is right $K$-invariant.  If $g$ has non-negative sectional curvature and, in particular, if $g$ is bi-invariant, then $g_1$ is also non-negatively curved.

We note that if $g$ is bi-invariant, $g_1$ can be described in the following way.  Using $g$, the Lie algebra of $G$, $\mathfrak{g}$, decomposes as $\mathfrak{g} = \mathfrak{k}\oplus \mathfrak{p}$ where $\mathfrak{k}$ is the Lie algebra of $K$.  For $X\in\mathfrak{g}$, write $X = X_\mathfrak{k} + X_\mathfrak{p}$.  Then, we have $g_1(X,Y) = g(X,\Phi_1 Y)$ where $\Phi_1(Y) = Y_\mathfrak{p} + \frac{t}{t+1} Y_\mathfrak{k}$.  Note that clearly $\Phi_1$ is invertible and $\Phi_1^{-1}(Y) = Y_\mathfrak{p} + \frac{t+1}{t}Y_\mathfrak{k}.$

In this situation, much is known about the $0$ curvature planes with respect to $g_1$.  The following result may be found in \cite{Es2}.

\begin{theorem}\label{1def}  Let $K\subseteq G$ be compact Lie groups and assume $(G,K)$ is a symmetric pair.  Let $g_1$ be Cheeger deformation of a bi-invariant metric $g$ in the direction of $K$ with parameter $t$.  Then a plane $\sigma = \operatorname{span}\{ \Phi_1^{-1}X, \Phi_1^{-1} Y\}$ has $0$ sectional curvature with respect to $g_1$ if and only if $$[X,Y] = [X_\mathfrak{k},Y_\mathfrak{k}] = [X_\mathfrak{p}, Y_\mathfrak{p}] = 0.$$

\end{theorem}

With respect to a bi-invariant metric, a plane $\operatorname{span}\{X,Y\}$ has $0$ curvature if and only if $[X,Y]=0$. We see that with the metric $g_1$ there are more constraints to satisfy, hence we expect $g_1$ to have fewer $0$ curvature planes.

\

Wilking's doubling trick, first observed in \cite{Wi}, arises from the simple observation that any biquotient $G\bq U$ defined by a subgroup $U\subseteq G\times G$ is canonically diffeomorphic to $\Delta G\backslash G\times G/ U$; a diffeomorphism is induced from the map $G\times G\rightarrow G$ sending $(g,h)$ to $g^{-1}h$.  Let $\pi_1,\pi_2:G\times G\rightarrow G$ be the two projection maps.  Consider subgroups $\pi_1(U)\subseteq H\subseteq G$ and $\pi_2(U)\subseteq K \subseteq G$.  Let $g_l$ denote the metric obtained by Cheeger deforming a bi-invariant metric in the direction of $H$, and let $g_r$ denote the metric obtained by deforming a bi-invariant metric in the direction of $K$.  Equipping $G\times G$ with the metric $g_l + g_r$, we see that $U$ acts by isometries, and hence the quotient map $G\times G\rightarrow \Delta G\backslash G\times G/U\cong G\bq U$ induces a new metric $g_0$ on $G\bq U$.  As before, the new metric will be non-negatively curved as well.

To understand when a tangent plane has $0$ curvature, Wilking \cite{Wi} proves that for each $(p,e)\in G$, the horizontal space with respect to $g_l+g_r$ of the $\Delta G\times U$ action is given by \begin{align*} \mathcal{H}_p = \Big\{\left(-\Phi_l^{-1} (Ad_{p^{-1}} X), \Phi_r^{-1} X\right)&: g(X, Ad_p u_1 - u_2) = 0 \\ & \text{ for all }(u_1,u_2)\in\mathfrak{u}\subseteq \mathfrak{g}\oplus\mathfrak{g}\Big\}.\end{align*}  Wilking then proves the following.

\begin{theorem}\label{Wilking} Fix a bi-invariant metric $g$ on $G$, and let $g_0$ be the metric on $G\bq U$ be constructed as above.  Let $\Phi_l$ be defined by $g(X,\Phi_l Y) = g_l(X,Y)$ and define $\Phi_r$ analogously.  Then, there is a $0$ curvature plane at the point $[p^{-1}]\in G\bq U$ if and only if there are two linearly independent vectors $$\hat{X_i} = (-\Phi_l^{-1}(Ad_{p^{-1}} X_i), \Phi_r^{-1}X_i) \in \mathcal{H}_p$$ with the property that $$\sec_{g_l}( \operatorname{span}\{ \Phi_l^{-1}(Ad_{p^{-1}} X_1), \Phi_l^{-1}(Ad_{p^{-1}} X_2)\}) = 0$$ and $$\sec_{g_r}(\operatorname{span}\{ \Phi_r^{-1} X_1, \Phi_r^{-1} X_2\}) = 0.$$

\end{theorem}

For our particular case, combining Theorems \ref{1def} and \ref{Wilking}, we have the following corollary.

\begin{corollary}\label{eqns}

Suppose $U\subseteq H\times K\subseteq G\times G$ and that $U$ acts on $G$ freely.  Assume $(G,K)$ and $(G,H)$ are both symmetric pairs.  Let $g_l$ and $g_r$ denote the Cheeger deformation of a bi-invariant metric in the direction of $H$ and $K$, respectively, and let $g$ be the metric induced on $G\bq U$ from the canonical diffeomorphism $\Delta G\backslash (G,g_l)\times (G,g_r)/U \rightarrow G\bq U$.  Finally, write $\mathfrak{g} = \mathfrak{k}\oplus\mathfrak{p} =\mathfrak{h}\oplus\mathfrak{q}$.   Then, there is a $0$ curvature plane at the point $[p^{-1}]\in G\bq U$ if and only if there are linearly independent vectors $X$ and $Y$ in $\mathfrak{g}$ with the property that

1.  $g_0(X, Ad_p u_1 - u_2) = g_0(Y, Ad_p u_1 - u_2) = 0$ for all $(u_1,u_2)\in\mathfrak{u}\subseteq \mathfrak{g}\oplus\mathfrak{g},$

2.  $[X,Y] = [X_{\mathfrak{k}}, Y_{\mathfrak{k}}] = [X_{\mathfrak{p}}, Y_{\mathfrak{p}}] = 0,$ and

3.  $[(Ad_{p^{-1}} X)_{\mathfrak{h}}, (Ad_{p^{-1}} Y)_{\mathfrak{h}}] = [(Ad_{p^{-1}} X)_{\mathfrak{q}}, (Ad_{p^{-1}} Y)_{\mathfrak{q}}] = 0$.

\end{corollary}

We note that if $X$ and $Y$ satisfy all three equations, then any pair of independent vectors having the same span as $X$ and $Y$ do as well.

\subsection{The topology of biquotients}
\label{top}

We will follow a method of Singhof \cite{Si1} and Eschenburg \cite{Es3} for computing the cohomology rings and characteristic classes of these biquotients.  We begin by letting $EG$ denote a contractible space on which $G$ acts freely, and hence $BG = EG/G$ will be the classifying space of $G$.  If the $U$ biquotient action on $G$ is free, then the projection $\pi:G\rightarrow G\bq U$ is an $U$-principal bundle, and is therefore classified by a map $\phi_U:G\bq U\rightarrow BU$.

Eschenburg \cite{Es3} has shown:

\begin{proposition}\label{commute} Suppose $\phi:U\rightarrow G\times G$ induces a free action of $U$ on $G$ and consider the fibration $\sigma:G\rightarrow B\Delta G\rightarrow BG\times BG$ induced by the diagonal inclusion $\Delta: G\rightarrow G\times G$.  There is a map $\phi_G:G\bq U\rightarrow B\Delta G$ so that the following is, up to homotopy, a pullback of fibrations.

\begin{diagram}
G & \rTo & G\bq U & \rTo^{\phi_U} & BU\\
& & \dTo^{\phi_G} & & \dTo^{Bf}\\
G & \rTo & B{\Delta G} &\rTo^{B\Delta} & BG\times BG\\
\end{diagram}

\end{proposition}

Using the fibration $G\rightarrow EG\rightarrow BG$ with $G = Sp(3)$, one sees that since $H^\ast(G) \cong \Lambda_\Z[x_3,x_7,x_{11}]$ with $\deg(x_i) = i$, then $H^\ast(BG) \cong \Z[\overline{x_3}, \overline{x_7}, \overline{x_{11}}]$ where $\deg(\overline{x_i})=i+1$ and $dx_i = \overline{x_i}$ in the spectral sequence for the fibration.

It is easily shown that in the spectral sequence for the fibration $G\rightarrow BG\rightarrow BG\times BG$ in Proposition \ref{commute}, $dx_i = \overline{x_i}\otimes 1-1\otimes \overline{x_i}.$  In particular, via naturality, computing the differentials in the spectral sequence for the biquotient is reduced to computing the map $Bf^\ast$ on cohomology.

The method for computing $Bf^\ast$ is due to Borel and Hirzebruch \cite{BH1}.  The idea is that if $T_U$ is a maximal torus in $U$, $T_G$ is a maximal torus in $G$, and $f:T_U\rightarrow T_G$, then there is an induced map $Bf^\ast:H^\ast(BT_G)\rightarrow H^\ast(BT_U)$ which is completely characterized by the weights of the representation $f:T_U\rightarrow T_G$.  We then identify $H^\ast(BG)$ as a subalgebra of $H^\ast(BT_G)$ and restrict $Bf^\ast$ to it.

More precisely, we first note that for a torus $T = T^n$ there is a natural isomorphism between $H^1(T)$ and $\operatorname{Hom}(\pi_1(T),\Z)$.  Moreover, if $\exp:\mathfrak{t}\rightarrow T$ denotes the exponential map, we can identify $\pi_1(T)$ with $\Gamma = \exp^{-1}(0)$.  This allows us to interpret roots and weights of a representation as elements of $H^1(T)$.  By using transgressions of generators of $H^1(T)$ as generators of $H^2(BT)$, we can interpret any weight as an element of $H^2(BT)$.  Also note that since the Weyl group of $G$ acts on $T$, it also acts on $H^\ast(BT)$.

Borel and Hirzebruch \cite{BH1} have shown:

\begin{theorem}\label{toruscomp}

Let $G$ be a compact Lie group with maximal torus $T$ and suppose $R$ is a ring with the property that $H^\ast(G;R)$ is an exterior algebra.  Then the map $i^\ast:H^\ast(BG;R)\rightarrow H^\ast(BT;R)$ induced from the inclusion $i:T\rightarrow G$ is injective, and the image consists of the Weyl group invariant elements of $H^\ast(BT;R)$.

\end{theorem}

For appropriate rings $R$, this theorem identifies $H^\ast(BG)$ as a sub-algebra of $H^\ast(BT_G)$.  Thus, we may compute $Bf^\ast:H^\ast(BT_G)\rightarrow H^\ast(BT_U)$ and then restrict to $H^\ast(BG)$.

For $G= Sp(3)$, we use $R = \mathbb{Z}$ and choose the maximal torus $$T_G = \{\diag(\exp(iy_1), \exp(i y_2),\exp(i y_3)): y_i\in \mathbb{R} \}.$$  Then, Theorem \ref{toruscomp} gives that
$$H^\ast(BG) \cong \Z[\sigma_1(\overline{y}^2), \sigma_2(\overline{y}^2), \sigma_3(\overline{y}^2)]\subseteq \Z[\overline{y}_1,\overline{y}_2, \overline{y}_3],$$ where the notation $\sigma_i(\overline{y}^2)$ denotes the elementary symmetric polynomials in the squares of the $\overline{y}_j$ variables.

\

Singhof \cite{Si1} has shown how to use this to compute the Pontryagin classes of the tangent bundle of $G\bq U$.

\

\begin{theorem}(Singhof)\label{charclass}

Suppose $U\subseteq G\times G$ defines a free biquotient action, then the total Pontryagin class of the tangent bundle of $G\bq U$ is given as $$p(G\bq U) = \phi_G^\ast\big(\Pi_{\lambda \in \Delta^+_G}(1+\lambda^2)\big)\phi_U^\ast\big(\Pi_{\rho\in\Delta^+_U}(1+\rho^2)\big)^{-1}$$ where $\Delta^+_G$ denotes the positive roots of $G$ and where $\phi_G^\ast$ and $\phi_U^\ast$ are the maps induced on cohomology.

\end{theorem}

In the case of $G= Sp(n)$, if $\lambda_i$ is the linear functional defined on $\mathfrak{t}_G$ by $\lambda_i(\diag(a_1,...,a_n)) = a_i$, then the positive roots are given by $\lambda_i + \lambda_j$ for $1\leq i\leq j\leq n$ and $\lambda_i - \lambda_j$ for $1 < i < j < n$.

\section{\texorpdfstring{Classification of effectively free biquotient actions of $Sp(1)\times Sp(1)$ on $Sp(3)$}{Classification of effectively free Sp(1)Sp(1) biquotient actions on Sp(3)}}
\label{Class}

\renewcommand*{\theHtable}{\arabic{table}} 

In this section, we classify all effectively free biquotient actions of $U = Sp(1)^2$ on $G = Sp(3)$.

\begin{theorem}\label{biqclass}  Up to equivalence, there are precisely $4$ homogeneous and $15$ inhomogeneous effectively free biquotient actions of $U= Sp(1)^2$ on $G = Sp(3)$.  Table \ref{table:emb} lists the image homomorphism $f:U\rightarrow G\times G$, with $(p,q)\in U$, defining these actions.

\begin{table}[H]

\caption{All biquotients of the form $Sp(3)\bq Sp(1)^2$}\label{table:emb}

\begin{center}

\begin{tabular}{|c|c|c|}

\hline

Name & Left factor image & Right factor image \\

\hline

$M_1$ & $I$ & $\diag(p,p,q)$\\

$M_2$ & $I$ & $\diag(p,q,1)$\\

$M_3$ & $I$ & $\diag(p,\phi_3(q) )$\\

$M_4$ & $I$ & $\diag(p,p,p)\cdot q'$\\

\hline

$N_1$ & $\diag(p,p,1)$ & $\diag(1,1,q)$\\

$N_2$ & $\diag(p,p,p)$ & $\diag(1,1,q)$\\

$N_3$ & $\diag(p,p,p)$ & $\diag(q,q,1)$\\

$N_4$ & $\diag(p,p,p)$ & $\diag(1,p,q)$\\

$N_5$ & $\diag(q,q,p)$ & $\diag(1,q,1)$\\

$N_6$ & $\diag(p,p,q)$ & $\diag(q,q,1)$\\

$N_7$ & $\diag(p,p,p)$ & $\diag(1,\phi_3(q))$\\

$N_8$ & $\diag(1,1,p)$ & $\diag(q,\phi_3(q))$\\

$N_9$ & $\diag(\phi_3(p),1)$ & $\diag(q,\phi_3(q))$\\

$N_{10}$ & $\diag(\phi_5(p))$ & $\diag(q,1,1)$ \\

$N_{11}$ & $\diag(\phi_3(q),1)$ & $\diag(p,p,p)\cdot q'$\\

$N_{12}$ & $\diag(p,1,1)$ & $\diag(q,q,q)\cdot p'$\\

$N_{13}$ & $p'$ & $\diag(p,q,q)$\\

\hline

$O_1$ & $p'$ & $\diag(q,q,q)$\\

$O_2$ & $p'$ & $\diag(q,1,1)$\\

\hline

\end{tabular}

\end{center}

\end{table}

\end{theorem}

In Table \ref{table:emb}, $q'$ denotes the image of $q$ under the canonical double cover $Sp(1)\rightarrow SO(3)$, where we view $SO(3)$ as a subgroup of $Sp(3)$.  The notation $\phi_i$ refers to the unique irreducible complex $i+1$-dimensional representation $\phi_i:Sp(1)\rightarrow SU(i+1)$, whose image, when $i$ is odd, is conjugate to a subgroup of the standard $Sp\left( \frac{i+1}{2}\right)\subseteq SU(i+1)$.

The $M$ biquotients are all homogeneous.  The $N$ biquotients are those where $U$ is isomorphic to $Sp(1)^2$, whereas for the $O$ biquotients, as well as $M_4$, $U$ is isomorphic to $Sp(1)\times SO(3)$.

As mentioned in the in Section 2, determining whether or not an action is effectively free, as well as determining the topology of these examples requires knowledge of the image of the maximal torus.  We, therefore, record these images in Table \ref{table:freelist}.

\begin{table}[h]

\caption{Image of the maximal torus of $Sp(1)^2$ in $Sp(3)^2$ for all effectively free actions}\label{table:freelist}

\begin{center}

\begin{tabular}{|c|c|c|}

\hline

Name & Left factor image & Right factor image\\

\hline

$M_1$ & $I$ & $\diag(z,w,w)$\\

$M_2$ & $I$ & $\diag(z,1,w)$\\

$M_3$ & $I$ & $\diag(z, z^3, w)$\\

$M_4$ & $I$ & $\diag(z\, w^2, z\, \overline{w}^2, z)$ \\

\hline

$N_1$ & $\diag(z,z,1)$ & $\diag(1,1,w)$\\

$N_2$ & $\diag(z,z,z)$ & $\diag(1,1,w)$\\

$N_3$ & $\diag(z, z, z)$ & $\diag(w,w,1)$\\

$N_4$ & $\diag(z,z,z)$ & $\diag(1,z,w)$\\

$N_5$ & $\diag(w,w,z)$ & $\diag(1,w,1)$\\

$N_6$ & $\diag(z,z,w)$ & $\diag(w,w,1)$\\

$N_7$ & $\diag(z,z,z)$ & $\diag(1,w, w^3)$\\

$N_8$ & $\diag(1,1,z)$ & $\diag(w,w,w^3)$\\

$N_9$ & $\diag(z,z^3,1)$ & $\diag(w,w^3, w)$\\

$N_{10}$ & $\diag(z,z^3, z^5)$ & $\diag(w,1,1)$\\

$N_{11}$ & $\diag(w,w^3 ,1)$ & $\diag(z\, w^2, z\, \overline{w}^2 , z)$\\

$N_{12}$ & $\diag(z,1,1)$ & $\diag(wz^2, w\overline{z}^2,w)$\\

$N_{13}$ & $\diag(z^2,\overline{z}^2,1)$ & $\diag(z,w,w)$\\

\hline

$O_1$ & $\diag(z^2, \overline{z}^2, 1)$ & $\diag(w, w, w)$\\

$O_2$ & $\diag(z^2, \overline{z}^2,1)$ & $\diag(w,1,1)$\\

\hline

\end{tabular}

\end{center}

\end{table}

To prove Theorem \ref{biqclass}, we begin by classifying all homomorphisms $U\rightarrow G$, from which one easily shows there are $484$ homomorphisms $U\rightarrow G^2$.  To determine which give rise to effectively free biquotient actions, we first classify all effectively free biquotient actions of $Sp(1)$ on $G$, finding precisely $17$.  Using symmetry considerations and the fact that the restriction of an effectively free action to a subgroup is still effectively free, we are then able to reduce our original $484$ homomorphisms down to a more manageable list.

With $\phi_i:Sp(1)\rightarrow SU(i+1)$ denoting the unique complex $i+1$-dimensional irreducible representation of $Sp(1)$, we let $\phi_{ij} = \phi_i\otimes \phi_j$ be the outer tensor product of $\phi_i$ and $\phi_j$.  It follows from Proposition \ref{sp1irrep} that $\phi_{ij}$ is orthogonal when $i$ and $j$ have the same parities and is symplectic otherwise.  As mentioned after Proposition \ref{sp1irrep}, Proposition \ref{symp} then implies that a sum of representations of $Sp(1)$ (respectively $U$) is symplectic if and only if each orthogonal $\phi_i$ (respectively $\phi_{ij}$) appears with even multiplicity.  From these observations it is easily seen that, up to equivalence, there are precisely $8$ homomorphisms $Sp(1)\rightarrow G$ given in Table \ref{table:sp1homo}.  Similarly, up to equivalence, there are $22$ homomorphisms $\rho:U\rightarrow G$, recorded in Table \ref{table:partition}.

For example, the $6$-dimensional representation of $U$, $\phi_{11} + \phi_{10}$ is not a symplectic representation, i.e., the image is not a subgroup of $Sp(3)\subseteq SU(6)$ because the orthogonal representation $\phi_{11}$ appears with odd multiplicity.  On the other hand, the representation $2\phi_{00} + \phi_{30}$ is symplectic, because $\phi_{30}$ is symplectic and $\phi_{00}$, though orthogonal, has even multiplicity.

\begin{table}[ht]

\caption{Homomorphisms from $Sp(1)$ into $G$}\label{table:sp1homo}

\begin{center}

\begin{tabular}{|c|c|}

\hline

Representation & Image of $T_{Sp(1)}$\\

\hline

$ 6\phi_0$ & \diag(1,1,1)\\

$4\phi_0 + \phi_1$ & $\diag(z,1,1)$\\

$2\phi_0 + 2\phi_1$ & $\diag(z,z,1)$\\

$3\phi_1$ & $\diag(z,z,z)$\\

$2\phi_0 + \phi_3$ & $\diag(1,z,z^3)$\\

$\phi_1 + \phi_3$ & $\diag(z,z,z^3)$\\

$\phi_5$ & $\diag(z,z^3,z^5)$\\

$2\phi_2$ & $\diag(z^2, \overline{z}^2,1)$\\

\hline

\end{tabular}

\end{center}

\end{table}

\begin{table}[ht]

\caption{Homomorphisms from $U$ into $G$}\label{table:partition}

\begin{center}

\begin{tabular}{|c|c||c|c|}

\hline

Representation & Image of $T_U$ & Representation & Image of $T_U$\\

\hline

$6\phi_{00}$ & $\diag(1,1,1)$& $2\phi_{00} + \phi_{03}$ & $\diag(w,w^3,1)$\\

$4\phi_{00} +\phi_{10}$ & $\diag(z,1,1)$ & $2\phi_{20}$ & $\diag(z^2, \overline{z}^2, 1)$\\

$4\phi_{00} + \phi_{01}$ & $\diag(w,1,1)$ & $2\phi_{02}$ & $\diag(w^2, \overline{w}^2, 1)$\\

$2\phi_{00} + 2\phi_{10}$ & $\diag(z,z,1)$ & $3\phi_{10}$ & $\diag(z,z,z)$\\

$2\phi_{00} + 2\phi_{01}$ & $\diag(w,w,1)$ & $3\phi_{01}$ & $\diag(w,w,w)$\\

$2\phi_{00} + \phi_{01} + \phi_{10}$ & $\diag(z,w,1)$ & $2\phi_{10} + \phi_{01}$ & $\diag(z,z,w)$\\

$\phi_{10} + \phi_{30}$ & $\diag(z,z,z^3)$ & $2\phi_{01} + \phi_{10}$ & $\diag(w,w,z)$\\

$\phi_{01} + \phi_{03}$ & $\diag(w,w,w^3)$ & $\phi_{50}$ & $\diag(z,z^3,z^5)$\\

$\phi_{10} + \phi_{03}$ & $\diag(z,w,w^3)$ & $\phi_{05}$ & $\diag(w,w^3,w^5)$\\

$\phi_{01} + \phi_{30}$ & $\diag(w,z,z^3)$ & $\phi_{12}$ & $\diag(z w^2, z \overline{w}^2, z)$\\

$2\phi_{00} + \phi_{30}$ & $\diag(z, z^3,1)$ & $\phi_{21}$ & $\diag(w z^2, w \overline{z}^2, w)$\\

\hline

\end{tabular}

\end{center}

\end{table}

To classify which pairs of these homomorphisms give rise to effectively free actions, we first classify which pairs of homomorphisms $Sp(1)\rightarrow G^2$ give rise to effectively free actions.

\begin{proposition}\label{sp1class}

Suppose $f=(f_1,f_2):Sp(1)\rightarrow G^2$ with $f_1$ nontrivial.  Then $f$ induces an effectively free biquotient action of $Sp(1)$ on $G$ if and only if either $f_2$ is trivial or, up to interchanging $f_1$ and $f_2$, $(f_1,f_2)$ is equivalent to one of the pairs in Table \ref{table:sp1pairs}.

\begin{table}[ht]

\caption{Homomorphisms $Sp(1)\rightarrow G^2$ defining inhomogeneous effectively free actions}\label{table:sp1pairs}

\begin{center}

\begin{tabular}{|c|c|}

\hline

$(4\phi_0 + \phi_1, 2\phi_0 + 2\phi_1)$ & $(4\phi_0 + \phi_1, 3 \phi_1)$\\  $(4\phi_0 + \phi_1, \phi_1 + \phi_3)$ &  $(4\phi_0 + \phi_1, \phi_5)$\\  $(4\phi_0 + \phi_1, 2\phi_2)$ & $(3\phi_1,2\phi_0 + 2\phi_1) $ \\ $(3\phi_1, 2\phi_0 + \phi_3)$ & $(3\phi_1, 2\phi_2)$ \\$ (\phi_1 + \phi_3, 2\phi_0 + \phi_3)$ & $(2\phi_2, 2\phi_0 + \phi_3)$ \\

\hline

\end{tabular}

\end{center}

\end{table}

\end{proposition}

\begin{proof}

Recall that a biquotient action defined by $(f_1, f_2)$ is effectively free if and only if for all $t \in T_{Sp(1)}$, if $f_1(t)$ is conjugate to $f_2(t)$, then $f_1(t) = f_2(t) \in Z(G)$. It immediately follows that $f_1$ and $f_2$ must be distinct, and that if either $f_1$ or $f_2$ is the trivial homomorphism, then the action is automatically free, which accounts for the $7$ homogeneous examples.

Recalling that two diagonal matrices in $Sp(3)$ are conjugate if and only if their entries are the same up to reordering and complex conjugation, the remaining $\binom{7}{2} = 21$ pairs of homomorphisms may be easily checked.  We present a few of the calculations.

The homomorphism $(4\phi_0 + \phi_1, 3\phi_1)$ gives rise to an effectively free action since the only way $\diag(z, z, z)$ and $\diag(z, 1, 1)$ can be conjugate is if $z=1$.

On the other hand, the homomorphism $(2\phi_0+\phi_3, 4\phi_0 + \phi_1)$ does not induce an effectively free action since the two matrices $\diag(z, z^3, 1) $ and $ \diag(z, 1, 1)$ are conjugate when $z$ is a nontrivial third root of unity.

\end{proof}

Since the restriction of any effectively free action to a subgroup is effectively free, we have the following simple corollary.

\begin{corollary}\label{restest} Suppose $f=(f_1,f_2):U = Sp(1)^2\rightarrow G^2$ defines an effectively free action of $U$ on $G$.  Then the restriction of $f$ to both factors of $U$, as well as to the diagonal $Sp(1)$ in $U$, must be equivalent to the homomorphisms in Proposition \ref{sp1class}.

\end{corollary}

Intuitively, this means that if $f=(f_1,f_2)$, with both $f_i$ in Table \ref{table:partition}, defines an effectively free action, then setting $z=1$, $w=1$, or $z=w$ must result in a pair of homomorphisms from Proposition \ref{sp1class}.

We can now prove Theorem \ref{biqclass}.

\begin{proof}

Let $f = (f_1,f_2):U\rightarrow G^2$ be any of the $484$ pairs of homomorphisms coming from Table \ref{table:partition}.  As remarked earlier, $(f_1,f_2)$ and $(f_2,f_1)$ define equivalent actions and the action defined by $(f_i,f_i)$ is never effectively free, so we reduce the number of pairs to check down to $\binom{22}{2} = 231$.  Of these, up to interchanging $z$ and $w$, only four entries of Table \ref{table:partition} contain both a $z$ and $w$, leading to the four homogeneous biquotients.  Therefore, we may assume that neither $f_1$ nor $f_2$ is trivial, which reduces the number to $210$ pairs.

We also note that interchanging $z$ and $w$, which corresponds to interchanging indices of the $\phi_{ij}$s, gives equivalent actions.  Taking this symmetry into account reduces the number of pairs to check down to $121$.  Since we have already classified effectively free actions of $Sp(1)$, we discard any pairs which do not have both a $z$ and a $w$ in them, leading to $89$ pairs.  Finally, Corollary \ref{restest} reduces this number down to $18$, of which only 3 do not give rise to effectively free actions.  We now provide computations for these 3, and also some prototypical computations for 2 that do give rise to effectively free actions.

The action given by $(\diag(z, 1, 1), \diag(zw^2, z\overline{w}^2, z))$ is not effectively free because setting $z=-1$ and $w=i$ gives conjugate matrices, neither of which is $\pm I$.

The action given by $(\diag(z, z^3, 1), \diag(w^2, \overline{w}^2, 1))$ is not effectively free because setting $z=i$ and $w=\sqrt{i}$ gives conjugate matrices, neither of which is $\pm I$.

Finally, the action given by $(\diag(z, z^3, 1), \diag(zw^2, z\overline{w}^2, z))$ is not effectively free because setting $z$ to be a fifth root of unity and $w=\sqrt{z}$ gives conjugate matrices, neither of which is $\pm I$.

On the other hand, the action given by $(\diag(z, 1, 1), \diag(z, z, w))$ is effectively free because if two such matrices are conjugate, we must have $z=1$, which then forces $w=1$.

Additionally, the action given by $(\diag(z, z^3, z), \diag(w, w^3, 1))$ is effectively free. If two such matrices are conjugate, then either $z=1$ or $z^3 = 1.$  If $z=1$, clearly both matrices must be the identity, so we may assume $z^3 = 1,$ and hence either $z=w$ or $z = \overline{w}$.  By replacing $w$ with $\overline{w}$, we obtain the same subgroup, hence the same action on $G$, so we can focus on the equation $z=w$.  By Proposition \ref{sp1class}, this action is free.

Continuing in this manner, it is easily seen that the other $13$ entries of Table \ref{table:freelist} give rise to effectively free actions.

\end{proof}

\section{New examples with quasi-positive curvature}
\label{curvature}

We now show the manifolds $N_1$ through $N_8$ in Table \ref{table:freelist} admit metrics of quasi-positive curvature.  

We begin with a bi-invariant metric $g_0(X,Y) = -\operatorname{Re}\operatorname{Tr}(XY)$ on $G = Sp(3)$ and perform a Cheeger deformation in the direction of $K = Sp(1)\times Sp(2)$, block embedded.  We let $g_1$ denote this new metric and define $\Phi$ as the linear map relating $g_1$ with the bi-invariant metric: $g(X, \Phi Y) = g_1(X,Y)$ for all $X,Y\in \mathfrak{sp}(3) = \mathfrak{g}$, the Lie algebra of $G$.  We decompose $\mathfrak{g}$ as $\mathfrak{g} = \mathfrak{k} \oplus \mathfrak{p}$, orthogonal with respect to a bi-invariant metric.  For a vector $X\in \mathfrak{g}$, we write $X = X_\mathfrak{k} + X_{\mathfrak{p}}$.

Using Wilking's doubling trick, we construct a new metric $g_2$ on $G$ as the submersion metric from the natural projection $(G\times G, g_1 + g_1)\rightarrow \Delta G\backslash G\times G\cong G$.  For any of the biquotients $N_1$ through $N_8$, the corresponding subgroup $U$ is a subgroup of $K\times K$, and hence acts isometrically on $(G,g_2)$.  This induces a metric on $G\bq U$ which we will show, using Corollary \ref{eqns}, is quasi-positively curved.  In fact, we will find points of quasi-positive curvature of the form $p = \begin{bmatrix} \cos \theta & \sin \theta & 0 \\ -\sin\theta & \cos\theta & 0 \\ 0 & 0 & 1\end{bmatrix}$.  For $N_1$ though $N_6$ will we will find an explicit description of the allowable $\theta$s, while for $N_7$ and $N_8$ we show that all $\theta$s with $0 < \theta < \frac{\pi}{4}$ define points of positive curvature..

As a first step, we note that if $X,Y \in \mathfrak{g}$ with $[X_\mathfrak{p}, Y_\mathfrak{p}] = 0$, then, $X_{\mathfrak{p}}$ and $Y_{\mathfrak{p}}$ are linearly dependent.  This follows from interpreting $\mathfrak{p}$ as the tangent space to $G/K = \mathbb{H}P^2$, which has positive sectional curvature with the induced metric.  By subtracting an appropriate multiple of $Y$ from $X$, we can and will assume in any application of Corollary \ref{eqns}, that $X_\mathfrak{p} = 0$.

Because the homomorphisms defining the $U$ action on $G$ for $N_1$ through $N_6$ involve tensor products of only the representations $\phi_1$ and $\phi_2$, while those defining $N_7$ and $N_8$ involve $\phi_3$, we will separate the rest of the argument into two subsections.

\subsection{\texorpdfstring{Quasi-positive curvature on $N_1, \ldots, N_6$}{Quasi-positive curvature on N1 to N6}}
\label{firstquasi}

We begin with the observation that for any point $p$ of the above form and for any subgroup $U$ for the biquotients $N_1$ through $N_6$, $Ad_p u_1 = u_1$ for any $(u_1,u_2)\in\mathfrak{u}\subseteq \mathfrak{g}\oplus\mathfrak{g}$.  This leads to the following lemma.

\begin{lemma}\label{lem1}  Suppose $X\in\mathfrak{g}$ satisfies condition 1 of Corollary \ref{eqns}.  Then $X_{33} = 0$ and $X_{11} = - X_{22}$ for $N_1, N_2, N_3,$ and  $N_6$ while $X_{11} = 0$ for $N_4$ and $N_5$.

\end{lemma}

\begin{proof}

Part 1 of Corollary \ref{eqns} gives $\operatorname{Re}\operatorname{Tr}(X, Ad_p u_1 - u_2) = \operatorname{Re}\operatorname{Tr}(X, u_1 - u_2) = 0$.  It is now easy to see that for $N_1$, $N_2$, $N_3$, and $N_6$, this equation implies $X_{11} = -X_{22}$ and for $N_4$ and $N_5$, this equation implies $X_{11} = 0$.

We work this out in detail only in the case of $N_4$, the other cases being similar.  In this case, $U = \left\{ \big(\diag(p,p,p), \diag(1,p,q)\big): p,q\in Sp(1)\right\}$ and therefore $\mathfrak{u} = \left\{ \big(\diag(r,r,r), \diag(0,r,s)\big): r,s\in\mathfrak{sp}(1) = \operatorname{Im}\mathbb{H} \right\}$.

We thus arrive at $ 0  = -\operatorname{Re}\operatorname{Tr}(X(u_1-u_2)) = -X_{11}(r) - X_{33}(r-s)$.  Setting $r = 0$ and letting $s$ vary implies $X_{33} = 0$.  Then, setting $s = 0$ and letting $r$ vary implies $X_{11} = 0$.

\end{proof}

With this, we can prove a stronger result on the form of $Y$.

\begin{lemma}\label{lem2}
Suppose $X, Y\in \mathfrak{g}$ with $X_{33} = Y_{33} = 0$ and that $X$ and $Y$ satisfy condition 2 of Corollary \ref{eqns}.  Then there are vectors $X',Y'\in \mathfrak{g}$ with $\operatorname{span}\{X', Y'\} = \operatorname{span}\{X,Y\}$ for which $X'_\mathfrak{p} = 0$ and $Y'_{\mathfrak{k}} = 0$.

\end{lemma}

\begin{proof}

The condition $[X_{\mathfrak{k}}, Y_{\mathfrak{k}}] = 0$ implies, since the $Sp(2)$ factor of $K$ is normal in $K$, that $[X_{\mathfrak{sp}(2)}, Y_{\mathfrak{sp}(2)}] = 0$.  Because $X_{33} = Y_{33} = 0$, we may interpret $X_{\mathfrak{sp}(2)}$ as an element of the tangent space to $Sp(2)/Sp(1) = S^7$.  Since this has positive curvature, $X_{\mathfrak{sp}(2)}$ and $Y_{\mathfrak{sp}(2)}$ are dependent.  Now, note that if $X_{\mathfrak{sp}(2)} = 0$, then Lemma \ref{lem1} implies that $X = 0$, so $X_{\mathfrak{sp(2)}} \neq 0$.  In particular, by subtracting an appropriate multiple of $X$ from $Y$, we may assume $Y_{\mathfrak{sp}(2)} = 0$.  Then Lemma \ref{lem1} implies $Y_{11} = 0$.
\end{proof}

We can now prove that $N_4$ and $N_5$ are quasi-positively curved.

\begin{theorem}\label{n4n5} For the biquotients $N_4$ and $N_5$, if $\theta$ is not an integral multiple of $\frac{\pi}{2}$, then every plane at the point $[p^{-1}]$ is positively curved.

\end{theorem}

\begin{proof}  We assume for a contradiction that there is a plane of $0$ curvature at the point $[p^{-1}]$.  Then there are vectors $X$ and $Y\in \mathfrak{g}$ satisfying the three conditions in Corollary \ref{eqns}.  Lemma 2 implies that we may assume $$X = \begin{bmatrix} 0& 0 & 0\\ 0& x_4 & x_5\\  0& -\overline{x_5} &0 \end{bmatrix} \text{ and } Y = \begin{bmatrix}0 & y_2 & y_3\\ -\overline{y_2} &0&0\\ -\overline{y_3} &0&0 \end{bmatrix}.$$   The condition $[X,Y]=0$ is equivalent to the pair of equations $y_2 x_5 = 0$ and $y_2 x_4 = y_3 \overline{x_5}$.  But these, coupled with the fact that $X\neq 0 \neq Y$, easily imply $y_2 =  x_5 = 0$.

Computing $Ad_{p^{-1}} X$ and $Ad_{p^{-1}}Y$, we arrive at

$$\begin{bmatrix} \sin^2\theta \, x_4 & -\cos\theta \sin\theta \, x_4 & 0 \\ -\cos\theta \sin\theta \, x_4&\cos^{2}\theta \, x_4 & 0  \\0  & 0 &0 \end{bmatrix} \text{ and } \begin{bmatrix} 0 &0 &\cos\theta\, y_3 \\0 & 0 & \sin\theta\, y_3 \\ - \cos\theta \,\overline{y_3} & - \sin\theta \,\overline{y_3} & 0 \end{bmatrix}.$$

Condition 3 of Corollary \ref{eqns} implies $[(Ad_{p^{-1}} X)_\mathfrak{k}, (Ad_{p^{-1}} Y)_\mathfrak{k}] = 0$, that is, that $x_4 y_3 \cos^2\theta \sin \theta = 0$.  Since $\theta $ is not a multiple of $\frac{\pi}{2}$, this equation implies $x_4 = 0$ or $y_3 = 0$.  Hence, $X = 0$ or $Y = 0$, so they do not span a plane.
\end{proof}

For $N_1$ through $N_3$ and $N_6$, we must restrict the form of $p$ further.

\begin{theorem}\label{n1n2n3n6} For $N_1$, $N_2$, $N_3$, or $N_6$, if $\theta$ is not an integral multiple of $\frac{\pi}{3}$ or $\frac{\pi}{4}$, then every plane at the point $[p^{-1}]$ is positively curved.

\end{theorem}

\begin{proof}

Assume for a contradiction that there is a $0$ curvature plane at the point $[p^{-1}]\in G\bq U$.  As in the proof of Theorem \ref{n4n5}, we may assume that there are two independent vectors of the form $$X=\begin{bmatrix} x_{1} & 0 & 0\\0 &-x_{1}& x_{5}\\0 & - \overline{x_5}& 0 \end{bmatrix} \text{ and } Y=\begin{bmatrix}0 & y_{2} & y_{3}\\ - \overline{y_2} & 0 & 0 \\ - \overline{y_3} & 0 & 0 \end{bmatrix}$$ satisfying all the conditions of Corollary \ref{eqns}.

The equation $[X,Y]=0$ is equivalent to the equations \begin{align} x_1 y_2 + y_2 x_1+y_3 \overline{x_5} &= 0\\  x_1 y_3 - y_2 x_5 &= 0. \end{align}

Equations (1) and (2), together with the fact that $X$ and $Y$ are non-zero, imply $x_1$ and $y_2$ are both non-zero.  We now see that $[(Ad_{p^{-1}}X)_\mathfrak{p}, (Ad_{p^{-1}}Y)_\mathfrak{p}]=0$ if and only if $$v_1 = \begin{bmatrix} 2\cos\theta\sin\theta x_{1} \\ -\sin\theta x_{5} \end{bmatrix} \text{ and } v_2 = \begin{bmatrix} \operatorname{Re}y_2 + (\cos^2\theta - \sin^2\theta)\operatorname{Im}y_2 \\ \cos\theta y_{3} \end{bmatrix}$$ are dependent over $\mathbb{R}$.

Since $\theta$ is not a multiple of $\frac{\pi}{4}$, $v_1$ and $v_2$ can only be dependent if $y_2$ is purely imaginary.  Thus, by rescaling $Y$, we may assume $y_2 = x_1$, which then implies $y_3  = \frac{1}{2}(\tan^2 \theta -1) x_5$.  Substituting this into (1) shows that $x_1 = 0$ if and only if $x_5 = 0$, so $x_5\neq 0$.  Then, substituting this into equation (2) gives $\frac{1}{2}(\tan^2 \theta -3) x_1 x_5 = 0$.  Since $\theta$ is not a multiple of $\frac{\pi}{3}$, this implies $x_1 x_5 = 0$, giving a contradiction.

\end{proof}

\subsection{\texorpdfstring{Quasi-positive curvature on $N_7$ and $N_8$}{Quasi-positive curvature on N7 and N8}}
\label{lastquasi}

Both biquotients in this section uses the homomorphism $\phi_3$, which we now describe on the Lie algebra level.

\begin{proposition}\label{pr41}
For $t = t_i + t_j + t_k \in \operatorname{Im}\mathbb{H}$, $$\phi_3(t) = \begin{bmatrix} 3t_i & \sqrt{3} (t_j+t_k) \\ \sqrt{3} (t_j+t_k) & 2(t_k-t_j)-t_i \end{bmatrix},$$ up to equivalence of representations.

\end{proposition}

\begin{proof}
First, one easily verifies that $\phi_3$ is a Lie algebra homomorphism, and thus $\phi(\mathfrak{sp}(1))$ is a sub-algebra of $\mathfrak{g}$.

The weights of the representation $\phi_3$ are $\pm 1$ and $\pm 3$.  Since the image of $i$ (after applying the natural inclusion $\mathfrak{sp}(2)\subseteq \mathfrak{su}(4)$) clearly has these eigenvalues, so, $\phi_3$ is the desired homomorphism.

\end{proof}

We let $Sp(1)_{\text{max}}$ denote the image of $Sp(1)$ in $Sp(2)$ under the homomorphism $\phi_3$.  Recall Berger \cite{Ber} showed that the Berger space $Sp(2)/Sp(1)_{\text{max}}$ is a normal homogeneous space of positive sectional curvature.

We note that, as in the case of $N_1$ through $N_6$, both $N_7$ and $N_8$ have the property that $Ad_p u_1 = u_1$ for any $(u_1,u_2)\in\mathfrak{u}$.

We will now begin to understand the structure of any $X$ and $Y$ satisfying the conditions of Corollary \ref{eqns}.  Recall that we are assuming $X = X_\mathfrak{k}$.  We note that for $N_7$, the fact that $X$ is orthogonal to $\mathfrak{u}_1$ immediately implies that $X_{\mathfrak{sp}(2)}\neq 0$.  Similarly, for $N_8$, the fact that $X$ is orthogonal to $\mathfrak{u}_2$ immediately implies that $X_{\mathfrak{sp}(2)} \neq 0$.  Using this, we show that $Y$ can be assumed to have a nice form.

\begin{lemma}\label{lemmax2}  Suppose $X$ and $Y$ satisfy conditions 1 and 2 of Corollary \ref{eqns}.  Then $X_{\mathfrak{sp}(2)}$ and $Y_{\mathfrak{sp}(2)}$ are linearly dependent.

\end{lemma}

\begin{proof} The condition $[X_\mathfrak{k},Y_\mathfrak{k}] = 0$ implies that $[X_{\mathfrak{sp}(2)}, Y_{\mathfrak{sp}(2)}] = 0$.

For $N_8$, the fact that $g_0(X, \mathfrak{u}_1) = g_0(Y,\mathfrak{u}_1)$ implies, via Lemma \ref{lem2}, that $X_{\mathfrak{sp}(2)}$ and $Y_{\mathfrak{sp}(2)}$ are linearly dependent.

For $N_7$, the condition $g_0(X, \mathfrak{u}_2) = g_0(Y,\mathfrak{u}_2) = 0$ allows us to interpret $X_{\mathfrak{sp}(2)}$ and $Y_{\mathfrak{sp}(2)}$ as elements of the tangent space of the Berger space $Sp(2)/Sp(1)_{\text{max}}$.  As this is known to have positive curvature \cite{Ber}, this implies $X_{\mathfrak{sp}(2)}$ and $Y_{\mathfrak{sp}(2)}$ are dependent.

\end{proof}

In particular, by subtracting an appropriate multiple of $X$ from $Y$, we can and will assume $Y_{\mathfrak{sp}(2)} = 0$.  Then, the argument used above to establish the fact that $X_{\mathfrak{sp}(2)}\neq 0$, when applied to $Y$, gives that $y_1 = 0$.

At this point, we know that if $X$ and $Y$ are linearly independent vectors in $\mathfrak{g}$ satisfying the conditions of Corollary \ref{eqns}, then we may assume without loss of generality  that $$X = \begin{bmatrix} x_1 & 0 & 0 \\ 0 & x_4 & x_5 \\ 0 & -\overline{x_5} & x_6\end{bmatrix} \text{ and } Y= \begin{bmatrix}0 & y_2 & y_3 \\ -\overline{y_2} & 0 & 0\\ -\overline{y_3} & 0 & 0 \end{bmatrix}.$$  Further, for $N_7$, orthogonality to $\mathfrak{u}_1$ implies $x_1 + x_4 + x_6 = 0$ while for $N_8$, it implies $x_6 = 0$.  Using Proposition \ref{pr41}, orthogonality to $\mathfrak{u}_2$ is equivalent to the following equations for $N_7$ and $N_8$ respectively.

\begin{table}[H]

\begin{center}

\begin{tabular}{clcl}

$(7a)$ & $(x_6)_i = 3(x_4)_i$ & $(8a)$ & $(x_1)_i = -3(x_4)_i$ \\

$(7b)$ & $(x_6)_j =  - \sqrt{3}(x_5)_j $ & $(8b)$ & $(x_1)_j = -2\sqrt{3}(x_5)_j$ \\

$(7c)$ & $(x_6)_k = \sqrt{3}(x_5)_k$ & $(8c)$ & $(x_1)_k = -2\sqrt{3}(x_5)_k$\\

\end{tabular}

\end{center}

\end{table}

The condition $[X,Y] = 0$  equivalent to the pair of equations \begin{align} x_1 y_2 - y_2 x_4 + y_3 \overline{x_5} &= 0\\ x_1 y_3-y_2 x_5-y_3x_6 &= 0.\end{align}

As in the proof of Theorem \ref{n1n2n3n6}, the condition $[(Ad_{p^{-1}} X)_\mathfrak{p}, (Ad_{p^{-1}} Y)_\mathfrak{p}] = 0$ is equivalent to the two vectors $$v_1 = \begin{bmatrix} \cos\theta \sin\theta(x_1-x_4)\\ -\sin\theta\, x_5\end{bmatrix} \text{ and } v_2 = \begin{bmatrix} \operatorname{Re}y_2 + (\cos^2\theta - \sin^2\theta) \operatorname{Im} y_2\\ \cos \theta\, y_3\end{bmatrix}$$ being linearly dependent over $\mathbb{R}$.

We now note that if $\theta$ is not an integral multiple of $\frac{\pi}{4}$, $v_2 \neq 0$.  For, if $\cos^2\theta \neq \sin^2\theta$ and  $\cos\theta \neq 0$, then $v_2 = 0$ if and only if $Y = 0$.  We now prove that for these $\theta$, $v_1\neq 0$.

\begin{lemma}\label{lemmax4}  If $\theta$ is not an integral multiple of $\frac{\pi}{4}$, and if $X$ and $Y$ satisfy all the conditions of Corollary \ref{eqns}, then $v_1\neq 0$.

\end{lemma}

\begin{proof}
Assume $v_1 = 0$.  The condition on $\theta$ implies $x_1 = x_4$ and $x_5 = 0$.  

For $N_7$, equations $(7b)$ and $(7c)$ imply $x_6$ only has an $i$ part.  The $i$ component of the condition $x_1 + x_4 + x_6 = 0$, together with $(7a)$, then implies $x_6 = 0$.  Then $0=x_1 + x_4 = 2x_1$, so $X=0$.

For $N_8$, equations $(8b)$ and $(8c)$ imply $x_1$ only has an $i$ part.  Then $(8a)$ implies that $x_1 = x_4 = 0$. \end{proof}

We can now show $N_7$ and $N_8$ are quasi-positively curved.

\begin{theorem}\label{n7n8}

For any $\theta$ between $0$ and $\frac{\pi}{4}$, $N_7$ and $N_8$ are positively curved at the point $[p^{-1}]$.

\end{theorem}

\begin{proof}

Assume for a contradiction that there is a $0$ curvature plane at $[p^{-1}]$.  Then there are independent vectors $X = X_\mathfrak{k},Y\in \mathfrak{g}$ satisfying the three conditions of Corollary \ref{eqns}.  Then Lemma \ref{lemmax2} implies $Y = Y_\mathfrak{p}$.  We form the vectors $v_1$ and $v_2$, which must be parallel and, by Lemma \ref{lemmax4}, non-zero.

Hence we write $v_2 = \lambda v_1$ with $\lambda\neq 0$ a real number.  Substituting these into (4.3), dividing out by $\lambda$, and simplifying gives the equation $$\frac{\cos\theta \sin\theta}{\cos^2\theta - \sin^2\theta} ( (x_1-x_4)^2 + [x_4,x_1]) = \frac{\sin\theta}{\cos\theta} |x_5|^2.$$  We first point out that the conditions on $\theta$ imply the coefficients are all positive.  Also, we note that $(x_1-x_4)^2$, being the square of a purely imaginary number, is a non-positive real number, while $|x_5|^2$ is a non-negative real number.  This implies $[x_4,x_1] = 0$ since it is purely imaginary.  Thus, we see that for $0< \theta < \frac{\pi}{4}$, the left hand side is then non-positive while the right hand side is non-negative.  It follows that both sides are $0$, that is, that $v_1 = 0$, contradicting Lemma \ref{lemmax4}.

\end{proof}

\section{\texorpdfstring{The topology of $Sp(3)\bq Sp(1)^2$}{The topology of Sp(3)/Sp(1)Sp(1)}}
\label{examtop}

\setcounter{table}{0}
\renewcommand*{\theHtable}{\arabic{section}.\arabic{table}} 

In this section we show that biquotients of the form $G\bq U = Sp(3)\bq Sp(1)^2$ are distinct up to diffeomorphism.  We also show they are not diffeomorphic to any previously known example of a quasi-positively curved manifold.  As a preliminary observation, note that the long exact sequence in homotopy groups associated to the fibration $U\rightarrow G\rightarrow G\bq U$ shows $\pi_2(G\bq U) \cong \pi_1(U)$.  Thus, we immediately see that, other than $M_4$, the $M$ and $N$ biquotients where $U$ is isomorphic to $Sp(1)^2$ are homotopically distinct from the $O$ biquotients, with $U\cong Sp(1)\times SO(3)$.  We will show that the $M$, $N$, and $O$ manifolds are all pairwise distinct up to diffeomorphism and that, with the possible exception of the pairs $(N_1,N_6)$ and $(M_1,N_3)$, the $M$ and $N$ examples are distinct up to homotopy as well.

\subsection{\texorpdfstring{The cohomology groups of $Sp(3)\bq Sp(1)^2$}{The cohomology groups of Sp(3)/Sp(1)Sp(1)}}
\label{examgroups}

Here, we will compute the cohomology groups of all biquotients of the form $G\bq U = Sp(3)\bq Sp(1)^2$ where $U$ is isomorphic to $Sp(1)^2$.  In particular, we will show that the order of $H^8(G\bq U)$ distinguishes most of these examples up to homotopy.

We now set up notation.  We use the maximal torus $$T_U =\{ (\exp(iz), \exp(iw)): z,w\in \mathbb{R} \}$$ of $U$ and likewise, let $$T_G = \{ \diag(\exp(i y_1), \exp(i y_2), \exp(i y_3)):y_i\in \mathbb{R}\}.$$  We then have isomorphisms $H^\ast(T_U) \cong \Lambda(z,w)$, $H^\ast(T_G) \cong \Lambda(y_1,y_2,y_3)$.  We also have an isomorphism $H^\ast(G) \cong \Lambda(x_3, x_7, x_{11}).$  We use the notation $f = (f_1,f_2):U\rightarrow G\times G$ to denote the embedding of $U$ into $G\times G$.

Then, using Theorem \ref{toruscomp}, we identify $H^\ast(BG\times BG)$ as the subalgebra of $$H^\ast(BT_{G\times G})\cong \Z[\overline{y}_i\otimes 1, 1\otimes \overline{y}_i] $$ given by $$H^\ast(B(G\times G)) \cong \Z[\sigma_i(\overline{y}^2)\otimes 1, 1\otimes \sigma_i(\overline{y}^2)]$$ for $i=1$, $2$, and $3$.  Further, in the spectral sequence associated to the fibration on the right in Proposition \ref{commute}, we have $dx_i = \sigma_i(\overline{y}^2)\otimes 1 - 1\otimes\sigma_i(\overline{y}^2)$.

Now, for each entry in Table \ref{table:freelist}, we compute the map $B_f^\ast$.  To do this, we first identify the map $f^\ast:H^\ast(T_{G\times G})\rightarrow H^\ast(T_U)$ and then, using the transgressions of the generators of the cohomology rings, translate this into $Bf^\ast:H^\ast(BT_{G\times G})\rightarrow H^\ast(BT_U)$.  For example we have, for $N_6$, $f_1^\ast(y_1) = f_1^\ast(y_2) = z$ and $ f_1^\ast(y_3) = w$ while  $f_2^\ast(y_1) = f_2^\ast(y_2) = w$ and $f_2^\ast(y_3) = 0$.  This implies \begin{align*} Bf^\ast\left(\sigma_1(\overline{y}^2)\otimes 1 - 1\otimes \sigma_1\left(\overline{y}^2\right)\right) &= (Bf_1^\ast - Bf_2^\ast)\left(\overline{y}_1^2 + \overline{y}_2^2 + \overline{y}_3^2\right) \\ &= 2\overline{z}^2 - \overline{w}^2.\end{align*}

Repeating this calculation and similar ones for $\sigma_2$ for all the manifolds in Table \ref{table:freelist}, we record the information in the following table.

\begin{table}[ht]

\caption{Calculation of $Bf^\ast$}\label{Bfcalc}

\begin{center}

\begin{tabular}{|c|c|c|}

\hline

Name & $(Bf_1^\ast - Bf_2^\ast)(\sigma_1)$ & $(Bf_1^\ast - Bf_2^\ast)(\sigma_2)$ \\

\hline

$M_1$  & $-\overline{z}^2-2\overline{w}^2 $ & $ -2\overline{z}^2\overline{w}^2 - \overline{w}^4 $ \\

$M_2$ & $-\overline{z}^2 - \overline{w}^2$  & $-\overline{z}^2\overline{w}^2$ \\

$M_3$ & $ -10\overline{z}^2-\overline{w}^2 $ & $ -9\overline{z}^4 - 10\overline{z}^2\overline{w}^2 $  \\

\hline

$N_1$ &  $2\overline{z}^2 - \overline{w}^2$ & $\overline{z}^4 $  \\

$N_2$ &  $ 3\overline{z}^2 - \overline{w}^2$ & $3\overline{z}^4$ \\

$N_3$ &  $ 3\overline{z}^2 - 2\overline{w}^2$ & $ 3\overline{z}^4 - \overline{w}^4$ \\

$N_4$ & $2\overline{z}^2-\overline{w}^2$ & $3\overline{z}^4 - \overline{z}^2\overline{w}^2$ \\

$N_5$ & $\overline{z}^2 + \overline{w}^2$ & $\overline{w}^4+2\overline{z}^2 \overline{w}^2$ \\

$N_6$ & $2\overline{z}^2 - \overline{w}^2$ & $\overline{z}^4+2\overline{z}^2\overline{w}^2 - \overline{w}^4$ \\

$N_7$ & $3\overline{z}^2-10\overline{w}^2$ & $3\overline{z}^4-9\overline{w}^4 $\\

$N_8$ & $  \overline{w}^2 - 11\overline{z}^2$ & $19\overline{z}^4$\\

$N_9$ & $10\overline{z}^2 - 11\overline{w}^2$ & $9\overline{z}^4-19\overline{w}^4$ \\

$N_{10}$ & $35\overline{z}^2 - \overline{w}^2$ & $259 \overline{z}^4$ \\

$N_{11}$ & $2\overline{z}^2 - 3\overline{w}^2$ & $-7\overline{z}^4 - 3\overline{w}^4$\\

$N_{12}$ & $-7\overline{z}^2-3\overline{w}^2$ & $-3\overline{w}^4-16\overline{z}^4$\\

$N_{13}$ & $\overline{z}^2 - 2\overline{w}^2$ & $-\overline{z}^4 - 2\overline{z}^2\overline{w}^2 - \overline{w}^4$\\

\hline

\end{tabular}

\end{center}

\end{table}

Now, in the spectral sequence associated to the fibration $G\rightarrow BG\rightarrow BG\times BG$, we have $x_3$ and $x_7$ totally transgressive with $dx_3 = \sigma_1 \otimes 1 -1\otimes \sigma_1$ and $dx_7 = \sigma_2\otimes 1-1\otimes \sigma_2$.  It follows by naturality that, in the spectral sequence for the fibration $G\rightarrow G\bq U\rightarrow BU$, we have $dx_3 = Bf^\ast(\sigma_1\otimes 1-1\otimes \sigma_1)$ and $dx_7 = Bf^\ast (\sigma_2\otimes 1-1\otimes \sigma_2)$.

By inspection of this spectral sequence, we see that $H^4(G\bq U)$ is isomorphic to $\mathbb{Z}^2/\langle dx_3\rangle$, where $\mathbb{Z}^2$ is generated by $\overline{z}^2$ and $\overline{w}^2$.  In particular, writing $dx_3 = \alpha\overline{z}^2 + \beta \overline{w}^2$, the homomorphism $\psi:\mathbb{Z}^2\rightarrow \mathbb{Z}$ with $\psi(s,t) = -\beta s + \alpha t$ is surjective with kernel $\langle dx_3\rangle$, so we see $H^4(G\bq U)$ is isomorphic to $\mathbb{Z}$.  By identifying the edge homomorphism with $\phi_U^\ast$, we also see that the induced map $\phi_U^\ast: H^4(BU)\rightarrow H^4(G\bq U)$ maps $\overline{z_i}^2$ to its image in $\mathbb{Z}^2/dx_3$.

Further inspection of the spectral sequence reveals that $H^8(G\bq U)$ is isomorphic to $\mathbb{Z}^3/X$, where $\mathbb{Z}^3$ is generated by $\overline{z}^4$, $\overline{w}^4$, and $\overline{z}^2\overline{w}^2$, and the subgroup $X$ is generated by $\overline{z}^2\, dx_3$, $\overline{w}^2\, dx_3$, and $dx_7$.

To understand $\mathbb{Z}^3/X$, we form a matrix $A_f$ whose rows are the coefficients of the elements $\overline{z}^2\, dx_3$, $\overline{w}^2\, dx_3$, and $dx_7$.  For example, for $N_6$, one has the matrix $\begin{bmatrix} 2 & 0 & -1\\ 0 & -1 & 2\\ 1 & -1 & 2\end{bmatrix}$ where the first column is the coefficient of $\overline{z}^4$, the second is the coefficient of $\overline{w}^4$, and the last the column is the coefficient of $\overline{z}^2\overline{w}^2$.  It is well known that the Smith normal form of this matrix completely determines the quotient $\mathbb{Z}^3/X$, and a direct computation shows the Smith normal form of any $A_f$ is of the form $\diag(1,1,\det(A_f))$.  It follows that $H^8(G\bq U) \cong \mathbb{Z}/ \det(A_f)$.  In particular, if $\det(A_f)\neq \pm \det(A_g)$, then the two biquotients associated to homomorphism $f$ and $g$ are homotopically distinct.  We record all of these determinants in Table \ref{table:det}.

\begin{table}[H]

\caption{The order of $H^8(G\bq U)$}\label{table:det}

\begin{center}

\begin{tabular}{|cccccccc|}

\hline

$M_1$ & $M_2$ & $M_3$ & $N_1$ & $N_2$ & $N_3$ & $N_4$ & $N_5$ \\ 

\hline

$3$ & $1$ & $91$ & $1$ & $3$ & $3$ & $1$ & $1$ \\

\hline

\hline

$N_6$ & $N_7$ & $N_8$ & $N_9$ & $N_{10}$ & $N_{11}$ & $N_{12}$ & $N_{13}$\\

\hline

$1$ & $219$ & $19$ & $811$ & $259$ & $75$ & $291$ & $9$\\

\hline

\end{tabular}

\end{center}

\end{table}

It follows from this that $M_2$, $N_1$, $N_4$, $N_5$, and $N_6$ could be diffeomorphic, and that $M_1$, $N_2$, and $N_3$ could be diffeomorphic, but all other pairs of biquotients $Sp(3)\bq Sp(1)^2$ are distinct up to homotopy.

\subsection{\texorpdfstring{Pontryagin classes of $Sp(3)\bq Sp(1)^2$}{Pontryagin classes of Sp(3)/Sp(1)Sp(1)}}
\label{examp1}

By computing the first Pontryagin class, we now show that all of the examples are distinct up to diffeomorphism.

We first handle the $M$ and $N$ cases.  By Theorem \ref{charclass}, we know that $$p(G\bq U) = \phi_G^\ast\big(\Pi_{\lambda \in \Delta^+_ G}(1+\lambda^2)\big)\phi_U^\ast\big(\Pi_{\rho\in\Delta^+_U}(1+\rho^2)\big)^{-1}$$ where $\Delta^+ G$ denotes the positive roots of $G$ and where $\phi_G^\ast$ and $\phi_U^\ast$ are the maps induced on cohomology.  For $G = Sp(3)$, the positive roots are $y_i\pm y_j$ for $1\leq i\leq j\leq 3$, and for $U = Sp(1)^2$, the positive roots are $2z$ and $2w$, so we have the formula \begin{align*} p_1(G\bq U) &= \phi_G^\ast\left(\sum_{1\leq i\leq j\leq 3} (\overline{y_i} \pm \overline{y_j})^2 \right) - \phi_U^\ast\left( 4\overline{z}^2 + 4\overline{w}^2 \right)\\ &= 8\phi_G^\ast\left(\overline{y_1}^2 + \overline{y_2}^2 + \overline{y_3}^2\right) - 4\phi_U^\ast(\overline{z}^2 + \overline{w}^2).\end{align*}

To compute $\phi_G^\ast(\overline{y_i})$, we note that $B\Delta^\ast( 1\otimes \overline{y_i}) = B\Delta^\ast (\overline{y_i}\otimes 1) = \overline{y_i}$.  Finally, commutativity of the diagram in Proposition \ref{commute} implies that $B\Delta \circ \Phi_G = \phi_U \circ B_f$, so $$\phi_G^\ast(\overline{y_i}) = \phi_U^\ast(B_f^\ast(1\otimes\overline{y_i})) = \phi_U^\ast(Bf^\ast(\overline{y_i}\otimes 1)).$$

For example, applying this to $N_6$, and using $1\otimes \overline{y_i}$, we have \begin{align*} p_1(N_6) &= 8\phi_G^\ast( \overline{y_1}^2 + \overline{y_2}^2 + \overline{y_3}^2) - 4\phi_U^\ast(\overline{z}^2 + \overline{w}^2) \\ &= \phi_U^\ast \big( 8 B_f^\ast( 1\otimes \overline{y_1}^2 + 1\otimes \overline{y_2}^2 + 1\otimes \overline{y_3}^2) - 4(\overline{z}^2 + \overline{w}^2)\big)\\ &=  \phi_U^\ast \big( 8(\overline{w}^2 + \overline{w}^2 + 0) - 4(\overline{z}^2 + \overline{w}^2)\big) \\ &= \phi_U^\ast(12 \overline{w}^2 -4 \overline{z}^2).\end{align*}

Using the map $\psi:\mathbb{Z}^2\rightarrow \mathbb{Z}$ from the previous subsection, we identify $\phi_U^\ast(12 \overline{w}^2 - 4 \overline{z}^2) = 1 \cdot -4 + 2\cdot 12 = 20\in \mathbb{Z} = H^4(G\bq U)$.

The results of similar calculations are given in Table \ref{table:Pontryagin}.

\begin{table}[H]

\caption{First Pontryagin class of the remaining examples.}\label{table:Pontryagin}

\begin{center}

\begin{tabular}{|c|ccccc|ccc|}

\hline

Manifold & $M_2$ & $N_1$ & $N_4$ & $N_5$ & $N_6$ & $M_1$ & $N_2$ & $N_3$ \\

$\pm p_1 \in \mathbb{Z}$ & $0$ & $4$ & $12$ & $8$ & $20$ & $4$ & $8$ & $28$\\

\hline

\end{tabular}

\end{center}

\end{table}

In particular, we see that all of these examples break into pairwise distinct diffeomorphism types.  Further, since $p_1$, taken mod $24$, is a homotopy invariant \cite{AH}, they are distinct up to homotopy except possibly for the two pairs $(N_1,N_6)$ and $(M_1,N_3)$.

\

We now distinguish the $O$ examples and $M_4$.  As mentioned at the beginning of section 5, these spaces all have second homotopy group isomorphic to $\mathbb{Z}/2\mathbb{Z}$, and are thus homotopically distinct from the $N$ family and other $M$ manifolds.  Because $H^\ast(SO(3))$ has torsion, we cannot directly use Theorem \ref{toruscomp} to compute $p_1(O_1)$ and $p_1(O_2)$.  However, by investigating the spectral sequence associated to the fibration $S^2\rightarrow BT^1\rightarrow BSO(3)$ induced from the inclusion $T^1\subseteq SO(3)$, it is easy to see that the induced map $H^4(BSO(3))\rightarrow H^4(BT^1)$ is an isomorphism.  Thus, the conclusion of Theorem \ref{toruscomp} still holds, at least in dimension 4, and this is enough to allow us to compute $p_1$.  The only other change is to note that for $SO(3)$, the root is not $2\overline{w}$, but rather $\overline{w}$.  Then, going through a similar calculation, we find $p_1(M_4) = -5$, $p_1(O_1) = 37$, and $p_1(O_2) = 7$.

\

Finally, we note that previously the only known $15$-dimensional examples with quasi-positive curvature were $M_1 = T^1\mathbb{H}P^2$, $M_2$, and $M_3$ as shown by Kerr and Tapp \cite{Ta1,KT}, $T^1 S^8$ as shown by Wilking \cite{Wi}, and an infinite family of biquotients of the form $U(5)\bq (S^1\times U(3))$ as shown by Kerin, Wilking, and Tapp \cite{Ke2,Wi,Ta1}.

Note that all of our new examples of quasi-positively curved manifolds are $3$-connected with $H^4$ isomorphic to $\mathbb{Z}$.  But $T^1 S^8$ is $6$-connected, and the $U(5)$ biquotients all have $H^2 = \mathbb{Z}$.

\bibliographystyle{plain}
\bibliography{bibliography}

\begin{thebibliography}{10}

\bibitem{AH}
M.~F. Atiyah and F.~Hirzebruch.
\newblock Riemann-{R}och theorems for differentiable manifolds.
\newblock {\em Bull. Amer. Math. Soc.}, 65:276--281, 1959.

\bibitem{Ber}
M.~Berger.
\newblock Les vari\'{e}t\'{e}s {R}iemanniennes homog\'{e}nes normales
  simplement connexes \'{a} courbure strictement positive.
\newblock {\em Ann. Scuola Norm. Sup. Pisa}, 15:179--246, 1961.

\bibitem{BH1}
A.~Borel and F.~Hirzebruch.
\newblock Characteristic classes and homogeneous spaces i.
\newblock {\em Amer. J. of Math.}, 80:458--538, 1958.

\bibitem{Ch1}
J.~Cheeger.
\newblock Some example of manifolds of nonnegative curvature.
\newblock {\em J. Diff. Geo.}, 8:623--625, 1973.

\bibitem{D1}
J.~DeVito.
\newblock The classification of simply connected biquotients of dimension 7 or
  less and 3 new examples of almost positively curved manifolds.
\newblock {\em Thesis, University of Pennsylvania}, 2011.

\bibitem{Es2}
J.~Eschenburg.
\newblock Freie isometrische aktionen auf kompakten {L}ie-gruppen mit positiv
  gekr$\ddot{\text{u}}$mmten orbitr$\ddot{\text{a}}$umen.
\newblock {\em Schriften der Math. Universit$\ddot{\text{a}}$t
  M$\ddot{\text{u}}$nster}, 32, 1984.

\bibitem{Es3}
J.~Eschenburg.
\newblock Cohomology of biquotients.
\newblock {\em Manu. Math.}, 75:151--166, 1992.

\bibitem{EK}
J.~Eschenburg and M.~Kerin.
\newblock Almost positive curvature on the gromoll-meyer 7-sphere.
\newblock {\em Proc. Amer. Math. Soc.}, 136:3263--3270, 2008.

\bibitem{FH}
W.~Fulton and J.~Harris.
\newblock {\em Representation Theory A First Course}.
\newblock Springer, 2004.

\bibitem{GrMe1}
D.~Gromoll and W.~Meyer.
\newblock An exotic sphere with nonnegative sectional curvature.
\newblock {\em Ann. Math.}, 100:401--406, 1974.

\bibitem{Ke2}
M.~Kerin.
\newblock Some new examples with almost positive curvature.
\newblock {\em Geometry and Topology}, 15:217--260, 2011.

\bibitem{Ke1}
M.~Kerin.
\newblock On the curvature of biquotients.
\newblock {\em Math. Ann.}, 352:155--178, 2012.

\bibitem{KT}
M.~Kerr and K.~Tapp.
\newblock A note on quasi-positive curvature conditions.
\newblock {\em Diff. Geo. and Appl.}, 34:63--79, 2014.

\bibitem{Ma}
Mal'cev.
\newblock On semisimple subgroups of {L}ie groups.
\newblock {\em Amer. Math. Soc. Translations}, 1:172--273, 1950.

\bibitem{On1}
B.~O'Neill.
\newblock The fundamental equations of a submersion.
\newblock {\em Michigan Math. J.}, 13:459--469, 1966.

\bibitem{PW2}
P.~Petersen and F.~Wilhelm.
\newblock Examples of {R}iemannian manifolds with positive curvature almost
  everywhere.
\newblock {\em Geom. and Top.}, 3:331--367, 1999.

\bibitem{Si1}
W.~Singhof.
\newblock On the topology of double coset manifolds.
\newblock {\em Math. Ann.}, 297:133--146, 1993.

\bibitem{Ta1}
K.~Tapp.
\newblock Quasi-positive curvature on homogeneous bundles.
\newblock {\em J. Diff. Geo.}, 65:273--287, 2003.

\bibitem{W}
F.~Wilhelm.
\newblock An exotic sphere with positive curvature almost everywhere.
\newblock {\em J. Geom. Anal.}, 11:519--560, 2001.

\bibitem{Wi}
B.~Wilking.
\newblock Manifolds with positive sectional curvature almost everywhere.
\newblock {\em Invent. Math.}, 148:117--141, 2002.

\end{thebibliography}

\end{document}